\documentclass[11pt]{amsart}
\usepackage{geometry}                
\usepackage{graphicx}
\usepackage{amssymb}
\usepackage{epstopdf}
\usepackage[all]{xy}
\usepackage{graphicx}
\usepackage{amssymb}
\usepackage{epstopdf}
\newtheorem{theorem}{Theorem}
\newtheorem{lemma}[theorem]{Lemma}
\newtheorem{prop}[theorem]{Proposition}
\newtheorem*{prop*}{Proposition}

\newtheorem*{cor*}{Corollary}
\newtheorem{corollary}[theorem]{Corollary}

\theoremstyle{definition}
\newtheorem{definition}[theorem]{Definition}
\newtheorem{example}[theorem]{Example}

\newtheorem{remark}[theorem]{Remark}
\newtheorem*{remark*}{Remark}
\newtheorem*{example*}{Example}

\newcommand{\Z}{{\mathbb {Z}}}
\newcommand{\R}{{\mathbb {R}}}

\newcommand{\secat}{{\sf {secat}}}
\newcommand{\RP}{{\mathbb {RP}}}
\newcommand{\h}{{\mathfrak h}}

\newcommand{\TC}{{\sf TC}}
\newcommand{\vv}{{\mathfrak  v}}
\newcommand{\e}{{\mathbf e}}
\newcommand{\VV}{{\mathcal  V}}

\title{Parametrized topological complexity of bundles of real projective spaces, I}
\author{Michael Farber}
\address{School of Mathematical Sciences \\
Queen Mary University of London\\
London, E1 4NS\\
United Kingdom}
\email{m.farber@qmul.ac.uk}

\author{Amit Kumar Paul}
\address{Stat-Math Unit, Indian Statistical Institute, Kolkata 700108, India}
\email{amitkrpaul23@gmail.com}

\author{Lucile Vandembroucq}
\address{Centro de Matem\'atica, Universidade do Minho, Braga,
Portugal}
\email{lucile@math.uminho.pt}

\begin{document}
\subjclass[2000]{Primary 55M30; Secondary 55P99}
%
%
%
%
%
%
\begin{abstract} Analysis of motion algorithms for autonomous systems operating under variable external conditions leads to the concept of parametrized topological complexity \cite{CFW}. In \cite{CFW}, \cite{CFW2} the parametrized topological complexity was computed in the case of the Fadell - Neuwirth bundle which is pertinent to algorithms of collision free motion of many autonomous systems in $\R^d$ avoiding collisions with multiple obstacles. The parametrized topological complexity of sphere bundles was studied in detail in \cite{FW}, \cite{FW2}, \cite{FP}. 
In this paper we make the next step by studying parametrized topological complexity of bundles of real projective spaces which arise as projectivisations of vector bundles. This leads us to new problems of algebraic topology involving theory of characteristic classes and geometric topology. 
We establish sharp upper bounds 
for parametrized topological complexity $\TC[p:E\to B]$ improving the general upper bounds. We develop algebraic machinery for computing lower bounds 
for $\TC[p:E\to B]$
based on the Stiefel - Whitney characteristic classes. Combining the lower and the upper bounds we compute explicitly many specific examples. 
\end{abstract}
\thanks{M. Farber and A. K. Paul were supported by a DMS-EPSRC grant}
\thanks{The research of L. Vandembroucq was partially financed by Portuguese
Funds through FCT (Funda\c c\~ao para a Ci\^encia e a Tecnologia) within
the Project UID/00013: Centro de Matem\'atica da Universidade do Minho
(CMAT/UM)}
\maketitle
\section{Introduction}

The concept of topological complexity is inspired by the robot motion planning problem of robotics \cite{F}, \cite{F2} it reflects the 
complexity of motion algorithms for autonomous systems. Parametrized topological complexity is a version of topological complexity adopted to situations when autonomous systems move under variable external conditions \cite{CFW}, \cite{CFW2}. Mathematically, parametrized topological complexity $\TC[p:E\to B]$ is a numerical invariant of a fibration $p:E\to B$ where the base $B$ represents the configuration space of the external conditions and $E$ is the configuration space of the compound system consisting of the controlled system and the external conditions. 

In this paper we analyse in detail the case when the bundle $p:E\to B$ is obtained from a vector bundle $\xi$ of rank $d+1$ over a CW complex $B$ by projectivisation, i.e. by taking the projective space of each fibre. The obtained {\it  projective bundle} is a locally trivial bundle with fibre $\RP^d$, the $d$-dimensional real projective space, and our goal is to express its parametrized topological complexity through the classical invariants of the vector bundle $\xi$. The Stiefel - Whitney classes of $\xi$ play a central role in our approach. 

The techniques we employ are slightly different in the cases $d=1$ and $d>1$, and we study these cases separately. 

In \S \ref{sec:3} we focus on the case $d=1$ -  bundles with fibre the circle. We show that the general upper bound of Proposition \ref{prop1} is never sharp for bundles with fibre $S^1$. The main results of this section, Corollary \ref{cor:9} and Theorem \ref{thm:10}, give more informative upper bounds. Combining these upper bounds with the lower bound given by Theorem 2 from \cite{FW} we are able to compute many examples and evaluate precisely the parametrized topological complexity. 
Material of this section is supported by Appendix \ref{sec:app1} which is related to a result of Ghys \cite{Ghys}. 

In section \S \ref{sec:2a} we study the cohomology algebras of projective bundles and derive the cohomological lower bound of Corollary \ref{cor:13}. This lower bound operates with algebraic structures built upon the cohomology algebra of the base $H^\ast(B;\Z_2)$ and the Stiefel - Whitney classes $w_i(\xi)$ of the vector bundle $\xi$. 

In the following section \S \ref{sec:5} we establish sharp upper bounds for parametrized topological complexity of projective bundles of dimension $d\ge 2$. The main result of this section is Theorem \ref{thm:sharp}; its proof uses Theorem \ref{maxsecatmanifold} proven in Appendix \ref{app:2}. Using the developed technique we are able to compute many examples in \S \ref{sec:6}. As a surprise we notice that in the case when the base $B=S^1$ is the circle and $\xi$ is a non-orientable vector bundle of rank 16 over $B$ we obtain a bundle of projective spaces $p:E \to B$ with fibre $\RP^{15}$ for which $$\TC[p:E\to B]=30.$$ This strongly contrasts with the fact that $$\TC(\RP^{15})=22.$$ 
This example illustrates that even in the case of a one-dimensional base the parametrized topological complexity can be significantly higher than the usual topological complexity of the fibre.

\section{Parametrized topological complexity}\label{sec:1}

In this section we briefly recall the notion of parametrized topological complexity which was introduced in \cite{CFW},  \cite{CFW2}. 
It is a generalisation  of the concept of topological complexity of robot motion planning problem, \cite{F}, \cite{F2}, which takes into account the state of external conditions. 

Let $X$ be a path-connected topological space viewed as the space of states of a mechanical system. The motion planning problem of robotics asks for an algorithm which takes as input an initial state and a desired state of the system, and produces as output a continuous motion of the system from the initial state to the desired state, see \cite{LV}. 
Given a pair of states $(x_0, x_1)\in X\times X$, the algorithm should produce a continuous path $\gamma: I\to X$ with 
$\gamma(0)=x_0$ and $\gamma(1)=x_1$, where $I=[0,1]$ denotes the unit interval representing time. 

Let $X^I$ denote the space of all continuous paths in $X$, equipped with the compact-open topology. 
The map $\pi: X^I \to X\times X$, where $\pi(\gamma) =(\gamma(0), \gamma(1))$,  is a fibration in the sense of Hurewicz.  A solution of the motion planning problem, {\it a motion planning algorithm}, is a section of this fibration, i.e. a map $s: X\times X \to X^I$ with $\pi\circ s ={\rm id}_{X\times X}$. 
If $X$ is not contractible, no section can be continuous, see \cite{F}.
{\it The topological complexity} of $X$ is defined to be the sectional category, or Schwarz genus, of the fibration 
$\pi: X^I \to X\times X$; notation: $\TC(X) =\secat(\pi)$. Explicitly, $\TC(X)$ is the smallest integer $k$ for 
 which there exists an open cover $$X\times X = U_0\cup U_1\cup \dots\cup U_k$$ such that 
  the fibration $\pi$ admits a continuous section $s_j : U_j \to X^I $ for each $j=0, 1, \dots, k$. 
 
In the parametrized setting initiated in \cite{CFW}, one assumes that the motion of the system is constrained 
by external conditions, such as obstacles or variable geometry of the containing domain. 
 The initial and terminal states of the system, as well as the motion between them, must live under the same external conditions.  
 This situation is modelled by a fibration 
 \begin{eqnarray}\label{eq:fibration}
 p: E \to  B
 \end{eqnarray} 
 with path-connected fibres, where the base $B$ is a topological space encoding the variety of external conditions. For $b \in B$, the fiber 
 $X_b = p^{-1}(b)$ is the space of achievable configurations of the system given the constraints imposed by $b$. 
 
 {\it A parametrized motion planning algorithm} takes as input the pairs of the initial and the terminal states of the system (consistent with identical external conditions), and produces a continuous  motion between them, achievable under the same external conditions. 
 In other words, the initial and the terminal states, as well as the path between them produced by the algorithm, all lie within the same fibre of the fibration (\ref{eq:fibration}). The state of the external conditions is implicitly also part of the input, as it coincides with the image of the initial and the final states under the projection (\ref{eq:fibration}).
 
 To define the notion of {\it  parametrized topological complexity} of the fibration $p: E \to B$ one needs to introduce  
 the associated fibration $\Pi: E^I_B \to E^2_B,$ where $E^2_B$ denotes the space of all pairs of configurations $(x_0, x_1)\in E^2=E\times E$ lying in the same fiber of $p$, i.e. $p(x_0)=p(x_1)$, while $E_B^I$ stands for the space of continuous paths in $E$ lying in a single fiber of $p$; the map $\Pi$ sends a path to the pair of its end points.
 \begin{definition}\label{def1}
 The parametrized topological complexity  $\TC[p:E\to B]$
 of the fibration $p: E \to B$ is defined as the sectional category of the associated fibration 
 \begin{eqnarray}\label{fib}
 \Pi: E_B^I \to E^2_B, \quad \Pi(\gamma)=(\gamma(0), \gamma(1)).\end{eqnarray}
In other words, 
$$\TC[p:E\to B]:=\secat[\Pi:E_B^I \to E^2_B],$$ is the minimal integer $k$ such that the space 
$E^2_B$ admits an open cover $$E^2_B=U_0\cup U_1\cup \dots\cup U_k$$ with the property that each set $U_i$ admits a continuous section 
$s_i:U_i\to E^I_B$ 
of $\Pi$, where $i=0, 1, \dots, k$.
\end{definition}

Note that $\Pi: E_B^I \to E^2_B$ is a Hurewicz fibration assuming that $p: E\to B$ is a Hurewicz fibration, see \cite{CFW2}, Proposition 2.1. 

The statement below is partially based on \cite{JGC}, it  explains the relevance of the notion of parametrized topological complexity to the problem of constructing motion planning algorithms:

\begin{prop}[Proposition 4.7 from \cite{CFW}]\label{prop3} If $p: E \rightarrow  B$ is a locally trivial fibration, and the spaces $E$ and $B$ are metrizable separable ANRs, then in Definition \ref{def1}, instead of open covers one may use arbitrary covers of $E^2_B$ or, equivalently, arbitrary partitions
$$E^2_B=F_0\sqcup F_1\sqcup \cdot \cdot \cdot \sqcup F_k, \quad F_i\cap F_j=\emptyset ,\quad i\not=j$$
admitting continuous sections $s_i : F_i \rightarrow  E_B^I$ of $\Pi$, where $i = 0, 1, . . . , k.$
\end{prop}

Next we mention the property of invariance with respect to fibre homotopy equivalence:

\begin{prop}[Proposition 5.1 from \cite{CFW}] If the fibrations $p: E \to B$ and $p' : E' \to B$ are fiberwise homotopy equivalent,
then $\TC[p: E \to B] = \TC[p' : E' \to B].$
\end{prop}

Finally we state the upper and the lower bounds for parametrized topological complexity. 

\begin{prop}\label{prop1} Let $p: E \rightarrow  B$ be a Hurewicz  fibration with fiber $X$, where the spaces 
$E, B, X$ are CW-complexes. Then the parametrized topological complexity $\TC[p:E\to B]$ satisfies the inequalities
\begin{eqnarray}\label{eq:3a}
\TC(X)\, \le \, \TC[p:E\to B] \, \le \,  \dim E^2_B =\dim B+2\dim X.\end{eqnarray}
\end{prop}
Inequalities (\ref{eq:3a}) follow directly from Definition \ref{def1} and the well-known properties of the sectional category. Compare Proposition 7.2 from \cite{CFW}. 

\begin{prop}[Proposition 7.3 from \cite{CFW}]\label{prop2}  Let $p: E \rightarrow  B$ be a fibration with path-connected fiber. Consider the diagonal map $\Delta : E \rightarrow  E^2_B$, where $\Delta (e) = (e,e)$ for $e\in E$. Then the parametrized topological complexity $\TC[p: E \rightarrow  B]$ is greater than or equal to the 
cup-length of the kernel $$\ker[\Delta^\ast  : H^\ast (E^2_B; R) \rightarrow  H^\ast (E; R)],$$ where $R$ is an arbitrary coefficient ring. 
\end{prop}

\section{Parametrized topological complexity of circle bundles}\label{sec:3}

In this section we consider locally trivial bundles with fibre $S^1$ (the circle) and describe their parametrized topological complexity. 
The results of this section serve two major purposes: they give a demonstrative illustration of the general method, and they are used later in this paper as the case of projective bundles of dimension 1. 

Let $$p: E\to B$$ be a locally trivial bundle with fibre the circle $S^1$, where the base $B$ is a finite CW-complex, {\it a circle bundle} for short.
We shall describe a cohomology class 
$$w(p) \in H^1(B;\Z_2)$$ associated with the bundle $p: E\to B$, which we call {\it the orientation class}. 
Each fibre $X_b=p^{-1}(b)$, where $b\in B$, is topologically a circle
and the homology group $H_1(X_b;\R)$ with real coefficients  is isomorphic to $\R$. The family of groups $${\mathcal H} =\{H_1(X_b;\R)\}_{b\in B}$$ 
forms a real line bundle over $B$, and the class $w(p)$ is defined as 
$
w(p) = w_1(\mathcal H),
$
where $w_1(\mathcal H)$ is the first Stiefel - Whitney class of $\mathcal H$. 
Let $\{U_i\}_{i\in I}$ be an open cover of $B$ such that the bundle $p: E\to B$ can be trivialised over each $U_i\subset B$ with the transition functions 
$$g_{ij}: U_i\cap U_j\to {\rm Homeo}(S^1).$$
Composing $g_{ij}$ with the obvious homomorphism ${\rm Homeo}(S^1)\to {\rm Aut}(H_1(S^1;\Z))=\{1, -1\}$ we obtain a $\{1, -1\}$-cocycle representing the orientation class $w(p)$. 

It is quite obvious that if the circle bundle $p: E\to B$ arises as the unit sphere bundle of a rank 2 vector bundle $\xi$ over $B$ then the orientation class 
$w(p)\in H^1(B;\Z_2)$ coincides with the first Stiefel - Whitney class $w_1(\xi)\in H^1(B;\Z_2)$. 

\begin{lemma} \label{lm:6a} 
The following properties of a locally trivial circle bundle $p: E\to B$ over a finite CW-complex $B$ are equivalent:
\begin{itemize}
\item[{(a)}] The class $w(p)=0\in H^1(B;\Z_2)$ vanishes;
\item[{(b)}]  Each fibre $X_b=p^{-1}(b)$, can be oriented with orientation continuous in $b\in B$; 
\item[{(c)}] The bundle $p: E\to B$ is isomorphic to an $SO(2)$-principal bundle;
\item[{(d)}]  The bundle $p: E\to B$ is isomorphic to the unit circle bundle of a complex line bundle over $B$. 
\end{itemize}
\end{lemma}
\begin{proof} If $w(p)=0$ then the real line bundle $\mathcal H$ is trivial and hence it admits a non-vanishing section. The section is a continuous function associating with each $b\in B$ a nonzero element in $H_1(X_b;\R)$, i.e. an orientation of the circle $X_b$. This shows that (a)$\implies$(b).

To show that (b)$\implies$(c) suppose that the fibres $X_b$ can be continuously oriented. 
The results of \S 4 from \cite{Ghys} imply that every orientable locally trivial circle bundle over a finite CW-complex admits $SO(2)$ structure group, 
i.e. it arises from a $SO(2)$-principal bundle. This means that for an $SO(2)$-principal bundle $P\to B$ one can identify $E$ with the space 
$P\times_{SO(2)}S^1$.
But clearly $$P\times_{SO(2)}S^1=P$$ and hence the circle bundle $p:E\to B$ coincides with the principal bundle $P\to B$. 

It is well-known that the classifying space $BSO(2)$ for principal $SO(2)$-bundles coincides with ${\mathbb {CP}}^\infty$, the universal space for complex line bundles. This shows equivalence between (c) and (d). 

The implications (d)$\implies$(a) and (d)$\implies$(b) are well-known: complex vector bundles are always orientable. 
\end{proof}

A locally trivial circle bundle satisfying the conditions of Lemma \ref{lm:6a} will be called {\it orientable}. 

The notion of genus of a cohomology class is well-known and appears in literature under several different names.
\begin{definition}\label{def:7}
The genus ${\mathfrak {genus}}(\alpha)$ of a cohomology class $\alpha\in H^1(B;\Z_2)$ is defined as the smallest integer $k\ge 0$ such that $B$ admits an open cover 
$B=U_0\cup U_1\cup \dots\cup U_k$ with the property that the restriction of the class $\alpha$ on each of the open sets $U_i$ vanishes, 
$\alpha|_{U_i}=0$, where $i=0, 1, \dots, k$. 
\end{definition}

It is well-known that the genus ${\mathfrak {genus}}(\alpha)$ coincides with the minimal integer $n$ such that the class $\alpha$ can be induced by a continuous map $B\to \RP^n$ from the generator 
$\beta\in H^1(\RP^n;\Z_2)$. 

Every nonzero cohomology class $\alpha\in H^1(B;\Z_2)$ determines a 2-fold regular covering space $p_\alpha: \tilde B\to B$ 
and, clearly, ${\mathfrak {genus}}(\alpha)$ coincides with the sectional category $\secat[p_\alpha: \tilde B\to B]$ of this covering. 

The upper and lower bounds for the genus
\begin{eqnarray}\label{eq:4a}
\h(\alpha) \le {\mathfrak {genus}}(\alpha)\le \dim B
\end{eqnarray} 
are well-known. Here $ \h(\alpha)$ denotes {\it the height of the class} $\alpha$, which is defined as the largest integer $m\ge 0$, such that $\alpha^m\not=0$. 

The following result uses the genus of the orientation class of a circle bundle to give a sharp upper bound for the parametrized topological complexity.

\begin{theorem}\label{thm:16} Let $p: E\to B$ be a locally trivial bundle with fibre the circle $S^1$ and 
let $w(p)\in H^1(B;\Z_2)$ be its orientation class. 
Then one has 
\begin{eqnarray}\label{eq:19}
1 \le \TC[p:E\to B]\le {\mathfrak {genus}}(w(p)) +1.
\end{eqnarray}
In particular, $\TC[p:E\to B]=1$ if the circle bundle $p:E\to B$ is orientable. 
\end{theorem}
An orientable circle bundle is a principal $S^1$-bundle (by Lemma \ref{lm:6a}) and our statement $\TC[p:E\to B]=1$ is consistent with Proposition 4.3 from \cite{CFW}. 

Combining the upper bound (\ref{eq:19}) with the inequality (\ref{eq:4a}) we obtain:
\begin{corollary}\label{cor:9}
For any circle bundle $p:E\to B$ over a finite CW-complex one has 
\begin{eqnarray}\label{eq:19a}
\TC[p:E\to B] \le \dim B +1. 
\end{eqnarray}
\end{corollary}
Note that the upper bound (\ref{eq:19a}) is by one stronger than the general upper bound of Proposition \ref{prop1}. 
We shall see below Example \ref{ex:12},  Example \ref{ex:13} and Example \ref{ex:14}, when the upper bound  (\ref{eq:19a}) is sharp. 

It would be of interest to compare the results about the sharp upper bounds for the parametrized topological complexity (Corollary \ref{cor:9} and  
Theorem \ref{thm:sharp})
with the results of \cite{Grant}. 

\begin{proof}[Proof of Theorem \ref{thm:16}] The inequality $\TC[p:E\to B]\ge 1$ follows from (\ref{eq:3a}) since $\TC(S^1)=1$. 
To prove the upper bound, we suppose that $B=U_0\cup U_1\cup \dots\cup U_k$ is an open cover with $w(p)|_{U_i}=0$ for $i=0, 1, \dots, k$, where $k={\mathfrak {genus}}(w(p))$. 
The restricted bundle $p: p^{-1}(U_i)\to U_i$ is orientable and we shall denote by $A_i\subset E^2_B$ the set  
$$A_i= \{(e, e')\in E^2_B; \, e\not= e', \, p(e)=p(e')\in U_i\}, \quad i=0, 1, \dots, k.$$ 
Besides, let $\Delta(E)\subset E^2_B$ be the diagonal. We obtain a cover
$$
E^2_B = \Delta(E)\cup A_0\cup A_1\cup \dots\cup A_k
$$
and (having in mind Proposition \ref{prop3}) we show below that over each of the sets of this cover one can construct a continuous section of the bundle $\Pi: E^I_B\to E^2_B$. 

A section $s: \Delta(E)\to E^I_B$ over the diagonal is obvious: we can set $s(e, e)\in E^I_B$ to be the constant path at $e\in E$. 

A section $s_i: A_i\to E^I_B$ for $i=0, 1, \dots, k$ can be defined as follows. By Lemma \ref{lm:6a} the bundle $p: p^{-1}(U_i)\to U_i$ is orientable and for $(e, e')\in A_i$ we define 
$s_i(e, e')(t)$ to be the path moving along the fibre in the positive direction starting from $e$ and ending at $e'$. This proves the inequality (\ref{eq:19}). 
\end{proof}

We can further strengthen the upper bound of Theorem \ref{thm:16} under an additional assumption that the base $B$ is a closed manifold:

\begin{theorem}\label{thm:10}
Let $p: E\to B$ be a locally trivial bundle with fibre $S^1$, where the base $B$ is a closed manifold of dimension $n$. If the $n$-th power of the orientation class $w(p)$ vanishes, i.e. 
$$w(p)^n=0\in H^1(B;\Z_2),$$ 
then 
\begin{eqnarray}
\TC[p:E\to B]\le \, n\, = \, \dim B.
\end{eqnarray}
\end{theorem}
\begin{proof}
By Theorem \ref{maxsecatmanifold} in Appendix II, see \S \ref{app:2},
we know that $w(p)^n=0$,  where $n=\dim B$, implies that ${\mathfrak {genus}}(p)< \dim B$. 
Our result now follows from the right inequality (\ref{eq:19}). 
\end{proof}

Finally we mention a result from \cite{FW},  which gives a lower bound for the parametrised topological complexity of circle bundles and complements Theorem \ref{thm:16}.

\begin{theorem}[Theorem 2 in \cite{FW}; see also Theorem 9.1 from \cite{FP}] 
\label{thm:11}
Let $\xi$ be a rank 2 vector bundle over a finite CW-complex $B$ and let $p:E\to B$ be the unit sphere bundle of 
$\xi$. Then one has
$$
\TC[p:E\to B] \ge \h(w_1(\xi)|w_2(\xi))+1,
$$
where $w_1(\xi)=w(p)$ and $w_2(\xi)$ are the Siefel-Whitney classes of $\xi$, with $w_i(\xi)\in H^i(B;\Z_2)$. The relative height $\h(w_1(\xi)|w_2(\xi))$ is defined as the maximal integer $k\ge 0$ such that the cohomology class $w_1(\xi)^k$ does not belong to the ideal of the algebra $H^\ast(B;\Z_2)$ generated by the class $w_2(\xi)$.
\end{theorem}

We illustrate the results of this section by several examples. 

\begin{example}\label{ex:12}
Consider the canonical line bundle $\eta$ over the real projective space $\RP^n$. Let $p:E\to B$ be the unit sphere bundle of the vector bundle $\xi=\eta\oplus \epsilon$ where $\epsilon$ is the trivial line bundle. The total Stiefel - Whitney class of $\xi$ is $1+\beta$, where $\beta\in H^1(\RP^n;\Z_2)$ is the generator, i.e. $w_1(\xi)=\beta$ and $w_2(\xi)=0$. Thus, we find that the relative height $\h(w_1(\xi)|w_2(\xi))$ equals $n$ and Theorem \ref{thm:11} gives 
$\TC[p:E\to B] \ge n+1$. On the other hand, Corollary \ref{cor:9} gives the opposite inequality implying $\TC[p:E\to B] = n+1$. 
\end{example}
\begin{example}\label{ex:13}
Consider the 2-dimensional torus $T^2$ and two real line bundles $\ell_1$ and $\ell_2$ over $T^2$ which have the Stiefel - Whitney classes $w_1(\ell_1)=a_1$ and 
$w_1(\ell_2)=a_2$, where $a_1, \, a_2\in H^1(T^2;\Z_2)$ are the standard generators. Let $\xi$ be the Whitney sum $\xi=\ell_1\oplus\ell_2$ and let 
$p:E\to B=T^2$ be the unit sphere bundle of $\xi$. The total Stiefel - Whitney class of $\xi$ is $$w(\xi)=(1+a_1)(1+a_2)= 1+ (a_1+a_2)+a_1a_2.$$
We see that $w_1(\xi)=a_1+a_2$ and $w_2(\xi)=a_1a_2$ and the relative height $\h(w_1(\xi)|w_2(\xi))=1$ equals 1. Applying Theorem \ref{thm:11} we get
$\TC[p:E\to B]\ge 2$. On the other hand, Theorem \ref{thm:10} gives the opposite inequality, and we finally obtain $\TC[p:E\to B]=2$.

\end{example}

\begin{example} \label{ex:14} The Hopf bundle $S^3\to S^2$ is orientable and hence $\TC[S^3\to S^2]=1$ according to Theorem \ref{thm:16}. 
Consider also the bundle $p: \RP^3\to S^2$ with fibre $\RP^1=S^1$ obtained from the Hopf bundle $S^3 \to S^2$ by factorising with respect to the antipodal 
involution on $S^3$. The obtained bundle is again orientable and hence we obtain $\TC[p: \RP^3\to S^2]=1$ as above. 
Example \ref{ex:26} contains further information on this circle bundle. 
\end{example}

\section{Projective bundles and cohomology algebras}\label{sec:2a}

\subsection{}\label{sec:21} Let $\xi$ be a vector bundle of rank $d+1>1$ over a CW-complex $B$. 
We shall denote by $E(\xi)$ the total space of $\xi$, and the bundle map is denoted $\xi: E(\xi)\to B$. 
The projectivisation of $\xi$ yields a locally trivial bundle 
\begin{eqnarray}\label{eq:bundle}
p:E\to B
\end{eqnarray} with fibre 
$\mathbb {RP}^d$, the real projective space of dimension $d$. Our aim is to understand the parametrized topological complexity
\begin{equation}
\TC[p:E\to B]
\end{equation}
(as defined in \S \ref{sec:1}) 
in terms of topological invariants associated with the original vector bundle $\xi$. 
We know that $\TC[p:E\to B]$ 
reflects the complexity of algorithms which take as input pairs of lines lying in the same fibre of $\xi$ and moving  the first line continuously towards the second line while staying in the same fibre, see \cite{CFW}, \cite{CFW2}.
\subsection{} An obvious first step is to compare $\TC[p:E\to B]$ with $\TC(\RP^d)$, the usual topological complexity of the fibre, which serves as a lower 
bound,  
\begin{eqnarray}\label{eq:compfibre}
\TC[p:E\to B]\ge \TC(\RP^d).
\end{eqnarray} 
From \cite{FTY} 
we know that for $d\not= 1, 3, 7$ one has
$$d+1\le \TC(\RP^d)\le 2d-1$$ 
and $\TC(\RP^d)= 2d-1$ when $d$ is a power of 2. One of the main results of \cite{FTY} states:
\begin{theorem}\label{fty}
For $d\not= 1, 3, 7$ the topological complexity $\TC({\mathbb {RP}}^d)$ equals the smallest integer $N$ such that the projective space ${\mathbb {RP}}^d$ admits a smooth immersion ${\mathbb {RP}}^d\to \mathbb R^N$. For $d=1, 3, 7$ one has $\TC({\mathbb {RP}}^d)=d$. 
\end{theorem}
Corollary 8.5 of \cite{FTY} states that for $d$ odd one has a stronger upper bound
\begin{eqnarray}\label{eq:upper}
\TC(\RP^d)\le 2d-\alpha(d)-k(d)
\end{eqnarray}
where $\alpha(d)$ is the number of ones in the dyadic expansion of $d$ and $k(d)$ depends only on the residue of $d$ modulo $8$ with $k(1)=0$, $k(3)=k(5)=1$ and $k(7)=4$. 
\begin{example} Taking $d=15$ we have $\alpha(15)=4$, $k(15)=4$ and hence inequality (\ref{eq:upper}) gives 
\begin{eqnarray}\label{eq:rp15}
\TC(\RP^{15})\le 22.
\end{eqnarray} 
This upper bound is in fact sharp, see Table 8.1 on page 1868 of \cite{FTY}. Note that \cite{FTY} uses unreduced definition of $\TC(X)$ which is 
 by one bigger than the definition we use in the present paper.   
\end{example}

\subsection{}\label{sec:43} Returning to the general setting of \S \ref{sec:21}, we may fix a scalar product on the vector bundle $\xi$ and consider 
the total space of {\it the unit sphere bundle} $\tilde E\subset E(\xi)$ associated with the vector bundle $\xi$. The canonical projection 
\begin{eqnarray}\label{eq:doublecover}
q: \tilde E\to E
\end{eqnarray} 
is a 2-fold covering and 
$E$ can be viewed as the quotient space of $\tilde E$ with respect to the antipodal involution $\tilde E\to \tilde E$, where $e\mapsto -e$ for $e\in \tilde E$. 
The space $$\tilde E^2_B=\tilde E\times_B\tilde E$$ is the space of all pairs of unit vectors lying in the same fibre of $\xi$. The group $G=\{1, -1\}$ acts on $\tilde E^2_B$ diagonally. We shall denote by
$$\widehat { E}^2_B=\tilde E^2_B/G$$
the quotient space.  Elements of the space $\widehat E^2_B$ are represented by pairs of unit vectors $(e_1, e_2)$ lying in the same fibre, i.e.  such that $pq(e_1)=pq(e_2)$, with the identification $(e_1, e_2)\sim (-e_1, -e_2)$; the equivalence class of $(e_1, e_2)$ will be denoted $[e_1, e_2]$. 
The map 
\begin{eqnarray}\label{hatq}
\hat q: \widehat E^2_B\to E^2_B
\end{eqnarray}
associates with $[e_1, e_2]$ the pair of lines $(\ell_1, \ell_2)$ containing the vectors $e_1$ and $e_2$. Clearly, $\hat q$ is a 2-fold covering. 
The covering transformation (the involution) $\tau: \widehat E^2_B\to \widehat E^2_B$ acts as follows: $$\tau[e_1, e_2]=[-e_1, e_2]=[e_1, -e_2].$$ 

\subsection{} Next we examine the cohomology algebra 
$H^\ast(E;\Z_2)$
of the total space of the projective bundle (\ref{eq:bundle}) with $\Z_2$ coefficients. 
Our goal is to relate $H^\ast(E;\Z_2)$ to the properties of the initial vector bundle $\xi: E(\xi)\to B$ and to its Stiefel - Whitney characteristic classes 
$w_i=w_i(\xi)\in H^i(B;\Z_2)$. 

The characteristic class of the double cover (\ref{eq:doublecover}),
\begin{equation}\label{eq:enhancement}
\vv\in H^1(E;\Z_2),
\end{equation}
is a nonzero class satisfying 
\begin{eqnarray}
q^\ast (\vv) =0\in H^1(\tilde E;\Z_2).
\end{eqnarray} 
\begin{definition}
We shall call the cohomology class $\vv\in H^1(E;\Z_2)$ {\it the enhancement of the bundle (\ref{eq:bundle}) determined by the bundle $\xi$}. 
\end{definition}
The cohomology class $\vv\in H^1(E;\Z_2)$ furnishes an additional structure on the projective bundle $p: E\to B$. 
The same projective bundle $p:E\to B$ may arise from different (not isomorphic) vector bundles $\xi$ and $\xi'$ and typically their enhancements will be different. 

\begin{lemma}\label{lm:proj} Let $\xi$ be a rank $d+1$ vector bundles of over a CW complex $B$. Let 
$\xi'=\xi\otimes L$, where $L$ is a line bundle over $B$. Then the projectivisation $p: E\to B$ of $\xi$ can be naturally identified with the projectivisation of $\xi'$. If 
$\vv, \vv' \, \in\,  H^1(E;\Z_2)$ denote the enhancements corresponding to the vector bundles $\xi$ and $\xi'$ correspondingly, then 
$$\vv'=\vv+p^\ast(w_1(L)),$$ where 
$w_1(L)\in H^1(B;\Z_2)$ is the first Stiefel - Whitney class of $L$.
\end{lemma} 
\begin{proof} 
Suppose that $V$ and $W$ are vector spaces and $\dim W=1$.  Any line $\ell\subset V$ and a nonzero element $w\in W$ determine a line $\ell\otimes w\subset V\otimes W$ and this line $\ell\otimes w$ is independent of the choice of $w\in W$. Thus, the projectivisations of $\xi$ and $\xi'=\xi\otimes L$ can be naturally identified. This fact is well known, see \cite{A}, \S 2.2.

To prove the second statement of Lemma \ref{lm:proj} consider the unit sphere bundle $\tilde E\to B$ of $\xi$ and the unit sphere bundle $\tilde S\to B$ of $L$. 
Both $\tilde E$ and $\tilde S$ carry natural action of the cyclic group of order two $G=\{1, -1\}$. A closed loop $\gamma:[0,1]\to E$ lifts to a closed loop in $\tilde E$ if and only if its evaluation $\langle \vv, [\gamma]\rangle =0$ vanishes. Similarly, the loop $p\circ \gamma$ in $B$ lifts to a closed loop in $\tilde S$ if and only if the evaluation 
$\langle w_1(L), p_\ast[\gamma]\rangle =0$ vanishes. 
The unit sphere bundle $\tilde E'$ of $\xi'=\xi\otimes L$ can be identified with the quotient of the space $\tilde E\times_B \tilde S$ with respect to the diagonal action of $G$, i.e. $\tilde E'=(\tilde E\times_B\tilde S)/G$. Hence a loop $\gamma: [0,1]\to E$ lifts to a closed loop in $\tilde E'$ if and only if  the lift of $\gamma$ into $\tilde E\to E$ and the lift of $p\circ\gamma: [0,1]\to B$ into $\tilde S\to B$ are either both closed or both not closed. Thus, we see that $\gamma$ lifts to a closed loop in $\tilde E'$ iff the evaluation of the class $\vv+p^\ast(w_1(L))\in H^1(E;\Z_2)$ on the homology class $[\gamma]\in H_1(E;\Z_2)$ vanishes. This shows that the enhancement corresponding to $\xi'$ equals $\vv +p^\ast(w_1(L)).$
\end{proof}

Next we describe the cohomology algebra of the projective bundle.

\begin{lemma}\label{lm:coh}
Consider the bundle of projective spaces $p: E\to B$ obtained as the projectivization of a vector bundle $\xi$ of rank $d+1$, where $d\ge 1$, over a CW-complex $B$. Then:
 
\begin{enumerate}
\item[{(a)}]
 every cohomology class $u\in H^k(E;\Z_2)$ has a unique representation in the form
\begin{eqnarray*}
u=a_k+a_{k-1}\cdot \vv+ a_{k-2}\cdot \vv^2+\dots + a_{k-d}\cdot \vv^d,
\end{eqnarray*}
where $a_i\in H^i(B;\Z_2)$ and the product $a_i\cdot \vv^{k-i}$ is understood as the cup-product $p^\ast(a_i)\cup \vv^{k-i}$. 
\item[{(b)}]  In the algebra $H^\ast(E;\Z_2)$ one has the following relation 
\begin{equation}\label{eq:rel}
\vv^{d+1} = w_{d+1} +w_{d}\cdot \vv +w_{d-1}\cdot \vv^2+ \dots+ w_1\cdot \vv^d,
\end{equation}
where $w_i=w_i(\xi)\in H^i(B;\Z_2)$ are the Stiefel - Whitney classes of the bundle $\xi$, $i=0, 1, \dots, d+1$.
\end{enumerate}
\end{lemma}
\begin{proof}
The restriction of the class $\vv$ on each fibre $\RP^d$ is the nonzero class $\beta\in H^1(\RP^d;\Z_2)$. 
Therefore the powers $\vv^i\in H^i(E;\Z_2)$ are nonzero for $i=1, 2, \dots, d$, since the restriction of the class $\vv^i$ onto the fibre $\RP^d$ equals $\beta^i\in H^i(\RP^d;\Z_2)$ and  is nonzero. Thus we see that Leray - Hirsch theorem is applicable. 
Statement  (a) follows directly from the Leray - Hirsch theorem. Statement (b) is one of the definitions of the Stiefel - Whitney classes, see \cite{Hu}, chapter 16. 
\end{proof}

\subsection{} As above, consider a vector bundle $\xi$ of rank $d+1$ over $B$ and its projectivisation $p: E\to B$. 
We shall deal with the bundle $E^2_B=E\times_B E$ with fibre $\RP^d\times\RP^d$ over $B$ associated with $\xi$. 
Recall that the total space $E^2_B$ is the space of pairs $(e, e')\in E^2$ such that $p(e)=p(e')$.  The space $E^2_B$ 
has two 2-fold coverings $\tilde E\times E\to E^2_B$ and $E\times \tilde E\to E^2_B$ where $\tilde E\subset E(\xi)$ is the unit sphere bundle of $\xi$. 
These two 2-folds coverings correspond to two cohomology classes 
\begin{eqnarray}\label{eq:vlvr}
\vv_L, \vv_R\in H^1(E^2_B;\Z_2),
\end{eqnarray} {\it the left and the right enhancements}. 
Similarly to Lemma \ref{lm:coh} we have:

\begin{lemma}\label{lm:coh1}
(a) Every cohomology class $u\in H^k(E^2_B;\Z_2)$ has a unique representation in the form
\begin{eqnarray}
u=\sum_{i\le d, \, \, j\le d} a_{ij}\cdot \vv_L^i\cdot \vv_R^j,
\end{eqnarray}
where $a_{ij}\in H^{k-i-j}(B;\Z_2)$.
(b) Each of the cohomology classes $\vv_L$ and $\vv_R$ satisfies the relation (\ref{eq:rel}). 
\end{lemma}
\begin{proof}
This follows from Leray - Hirsch theorem. 
\end{proof}

%

One may express the characteristic class of the covering (\ref{hatq}) in terms of the enhancements $\vv_L$, $\vv_R$ as explained below. 

\begin{lemma}\label{lm:21}
Let $\xi$ be a vector bundle of rank $d+1$, with $d\ge 2$, over a finite CW-complex $B$ and let $p:E\to B$ be the bundle with fibre $\RP^d$ obtained as projectivization of $\xi$. Then the  the characteristic class of the 2-fold covering $\hat q: \widehat{E^2_B}\to E^2_B$ (defined in \S \ref{sec:43}) 
coincides with $\alpha=\vv_L+\vv_R\in H^1(E^2_B;\Z_2)$. 
\end{lemma}
\begin{proof} The total space of the unit sphere bundle $\tilde E\to B$ has an action of $G=\{1, -1\}$ and the total space of the bundle 
$\tilde E^2_B=\tilde E\times_B\tilde E$ carries an action of $G\times G$. The classes $\vv_L$ and $\vv_R$ are characterised by the following property. If $\gamma: [0,1]\to E^2_B$ is a closed loop then its lift into $\tilde E\times_B E\to E^2_B$ is a closed loop iff and only if the value of the class 
$\vv_L$ on the homology class of $\gamma$ is $0$. Similarly, the lift of $\gamma$ into the covering $E\times_B\tilde E\to E^2_B$ is a closed loop iff $\vv_R(\gamma)=0$. 
Now, a loop $\gamma$ in $E^2_B$ lifts into the covering $\widehat{E^2_B}\to E^2_B$ if and only if its lifts with respect to the coverings 
$\tilde E\times_B E\to E^2_B$ and $E\times_B\tilde E\to E^2_B$ are either both closed or both not closed, i.e. when either 
$\vv_L(\gamma)=0$ and $\vv_R(\gamma)=0$, or when $\vv_L(\gamma)=1$ and $\vv_R(\gamma)=1$. We see that a loop 
$\gamma$ in $E^2_B$ lifts into the covering $\widehat{E^2_B}\to E^2_B$ iff $(\vv_L+\vv_R)(\gamma)=0$. 
\end{proof}

\begin{corollary}\label{cor:11}
The homomorphism $$\Delta^\ast: H^\ast(E^2_B;\Z_2)\to H^\ast(E;\Z_2),$$ induced on $\Z_2$-cohomology by the 
diagonal map, $\Delta: E\to E^2_B$, satisfies 
\begin{eqnarray}
\Delta^\ast(\vv_L)= \Delta^\ast(\vv_R)=\vv.
\end{eqnarray} 
Moreover, the sum $u=\vv_L+\vv_R$ is the only nonzero cohomology class in $H^1(E^2_B;\Z_2)$ satisfying $\Delta^\ast(u)=0$. 
\end{corollary}
\begin{proof}
For $i=1, 2$ consider the projection $\pi_i: E^2_B \to E$ onto the first or the second factor correspondingly. Then clearly $\pi^\ast_1(\vv)=\vv_L$ and $\pi_2^\ast(\vv)=\vv_R$ as follows from the definitions. Since $\pi_1\circ \Delta= 1_E=\pi_2\circ \Delta$, we obtain $\Delta^\ast(\vv_L)=\Delta^\ast(\pi^\ast_1(\vv))=(\pi_1\circ\Delta)^\ast(\vv) = \vv$ and similarly $\Delta^\ast(\vv_R)=\vv$. 

To prove the second statement consider a class $u\in H^1(E^2_B;\Z_2)$. By Lemma \ref{lm:coh1}  class $u$ has a unique representation in the form 
$u=a+\alpha_L\cdot \vv_L+ \alpha_R\cdot \vv_R$ where $a\in H^1(B;\Z_2)$ and $\alpha_L, \alpha_R\in \Z_2$. Then $\Delta^\ast(u) = a+ (\alpha_L+\alpha_R)\cdot \vv$. We see that $\Delta^\ast(u)=0$ iff $a=0$ and $\alpha_L=\alpha_R$. This shows that the sum $u=\vv_L+\vv_R$ is the only nonzero class in $H^1(E^2_B;\Z_2)$ 
satisfying $\Delta^\ast(u)=0$. 
\end{proof}

Recall that {\it the height} $\h(u)$ of a cohomology class $u$ is defined as the maximal integer $k$ such that $u^k\not=0$. 

\begin{corollary}\label{cor:13}
For the projective bundle $p:E\to B$ arising as projectivisation of  a vector bundle $\xi$ of rank $d+1>1$ the following lower bound holds
\begin{eqnarray}
\TC[p:E\to B] \ge \h(\vv_L+\vv_R),
\end{eqnarray}
where $\vv_L, \, \vv_R\in H^1(E^2_B)$  are the left and right enhancements defined above. 
\end{corollary}
\begin{proof} This follows from Proposition \ref{prop2} and Corollary \ref{cor:11}. 
\end{proof}
\begin{example} Let $\eta$ be the standard line bundle over the projective space $B=\RP^n$. Consider the rank 2 vector bundle $\xi=\eta\oplus \epsilon$, where $\epsilon$ is the trivial line bundle. We shall examine the projectivisation $p: E\to B=\RP^n$ of $\xi$. The full Stiefel - Whitney class of $\xi$ is 
$w(\xi)=1+\beta$, where 
$\beta\in H^1(\RP^n;\Z_2)=H^1(B;\Z_2)$ is the generator. Applying Lemma \ref{lm:coh1} we obtain that the cohomology algebra $H^\ast(E^2_B;\Z_2)$ has 3 polynomial generators: $\beta, \vv_L, \vv_R$ satisfying the following defining relations:
$$\beta^{n+1}=0, \quad \vv_L^2 =\beta\cdot \vv_L, \quad \vv_R^2 =\beta\cdot \vv_R.$$
We see that
$$
(\vv_L+\vv_R)^2=\beta \cdot (\vv_L+\vv_R),
$$
and it follows by induction that 
$$
(\vv_L+\vv_R)^k=\beta^{k-1} \cdot (\vv_L+\vv_R), \quad k\ge 2.
$$
Thus, the power $$(\vv_L+\vv_R)^{n+1}=\beta^n\cdot (\vv_L+\vv_R)$$ is the highest nonzero power of the class $\vv_L+\vv_R$, and Corollary \ref{cor:13} gives in this case:
\begin{equation}\label{eq:15}
\TC[p:E\to B]\ge n+1. 
\end{equation}
Corollary \ref{cor:9} (see also Theorem \ref{thm:sharp}) give an opposite inequality $\TC[p:E\to B]\le n+1$. Thus, we conclude that in this instance 
$\TC[p:E\to B]= n+1$. 
\end{example}
\begin{example} This example is a variation of the previous one. 
Let $\eta$ be the standard line bundle over the projective plane $B=\RP^2$. 
Consider the rank 3 vector bundle $\xi=\eta\oplus \epsilon\oplus \epsilon$ where $\epsilon$ is the trivial line bundle
and its projectivisation $p: E\to B=\RP^2$. It is a locally trivial bundle with fibre $\RP^2$ over $\RP^2$. 
The full Stiefel - Whitney class of $\xi$ is 
$w(\xi)=1+\beta$, where 
$\beta\in H^1(\RP^2;\Z_2)=H^1(B;\Z_2)$ is the generator. 
By Lemma \ref{lm:coh1} the cohomology algebra $H^\ast(E^2_B;\Z_2)$ has 3 polynomial generators: $\beta, \vv_L, \vv_R$ 
satisfying the following relations:
$$\beta^{3}=0, \quad \vv_L^3 =\beta\cdot \vv_L^2, \quad \vv_R^3 =\beta\cdot \vv_R^2.$$
We see that
\begin{eqnarray}\label{eq:16}
\vv_L^4=\beta^2\cdot \vv_L^2, \quad \vv_L^5=0 \quad \mbox{and similarly} \quad \vv_R^4=\beta^2\cdot \vv_R^2, \quad \vv_R^5=0.
\end{eqnarray}
Hence, 
$$
(\vv_L+\vv_R)^5= \vv_L^4\cdot \vv_R+\vv_L\cdot \vv_R^4=\beta^2\cdot \vv_L^2\cdot \vv_R+\beta^2\cdot \vv_L\cdot \vv_R^2 \, \not=\, 0
$$
is nonzero. On the other hand, 
$$
(\vv_L+\vv_R)^6=0,
$$
as one easily checks using the relations (\ref{eq:16}). Corollary \ref{cor:13} gives 
\begin{equation}\label{eq:17}
\TC[p:E\to B]\ge 5. 
\end{equation}
The upper bound of of Proposition \ref{prop1} gives $\TC[p:E\to B]\leq 6$; however, applying Theorem \ref{thm:sharp} we get 
$\TC[p:E\to B]=5$. 
\end{example}
\begin{example}\label{ex:26}
Consider the bundle $p: \RP^3\to S^2$ with fibre $\RP^1=S^1$ obtained from the Hopf bundle $S^3\to S^2$ by factorisation with respect to the 
antipodal involution, see Example \ref{ex:14}. In this example we have 
$
\vv_L^2=w$
 and
 $\vv_R^2=w
$
for $w\in H^2(S^2;\Z_2)$ with $w\not=0$. Therefore, we see that $(\vv_L+\vv_R)^2=\vv_L^2 +\vv_R^2 =0$ and $$\h(\vv_L+\vv_R)=1.$$ This is consistent with the result 
$\TC[p: \RP^3\to S^2]=1$ of Example \ref{ex:14}. 
Note however that $$\h(\vv_L)=\h(\vv_L)=3.$$ 
This example shows that the height of the sum $\vv_L+\vv_R$ can be {\it smaller} than the height of the classes $\vv_L$ and $\vv_R$. 

In fact, one may show that the difference $\h(\vv_L)- \h(\vv_L+\vv_R)$ can be arbitrarily large. 

\end{example}


\section{Sharp upper bound}\label{sec:5}

In this section we establish upper bounds for bundles of projective spaces which improve the general upper bound given by Proposition \ref{prop1}. 
We already considered the case of circle bundles (see Corollary \ref{cor:9} and Theorem \ref{thm:10}) and therefore in this section we shall focus on the case of bundles with fibre $\RP^d$ where $d\ge 2$. We restrict our attention to the projective bundles which arise as projectivizations of vector bundles. 
In some sense, our aim in this section is to find a generalisation of 
Theorem 1 from \cite{CF}
to the parametrized setting,

Let $\xi$ be a vector bundle of rank $d+1$ over a CW-complex $B$.  
We denote by $E(\xi)$ the total space of the bundle $\xi$ and by 
$\xi: E(\xi)\to B$ the projection. The unit sphere bundle of $\xi$ is denoted $\xi: \tilde E\to B$. 
The projectivisation 
$p: E\to B$ of $\xi$ is a locally trivial bundle with fibre $\RP^d$. Here $E$ is the quotient of $\tilde E$ with respect to the antipodal involution. 
Our aim is to establish an upper bound for the parametrized topological complexity
$\TC[p: E\to B]$ strengthening the general upper bound given by Proposition \ref{prop1}. The following is the main result of this section:

 \begin{theorem} \label{thm:sharp}
 Let $\xi$ be a vector bundle of rank $d+1$ over a closed manifold $B$ of dimension $n=\dim B$ and  let $p: E\to B$ be the bundle of real projective spaces 
 $\RP^d$ obtained as the projectivisation of $\xi$. 
 Then $\TC[p:E\to B]\le n+2d-1$.
 \end{theorem}
 This upper bound is by one better than the general upper bound of Proposition \ref{prop1}.

The proof of Theorem \ref{thm:sharp} will be completed at the end of this section. Note that the case $d=1$ of Theorem \ref{thm:sharp} is covered by Corollary \ref{cor:9}. Therefore we shall assume below that $d\ge 2$. 


%
%

The bundle $\Pi: E^I_B\to E^2_B$, which appears in Definition \ref{def1}, has as its fibre over a pair $(e, e')\in E^2_B$ the space $F_{(e, e')}$ of all paths $\gamma: I\to E$ satisfying $\gamma(0)=e$, $\gamma(1)=e'$, and such that the path $p\circ\gamma: I\to B$ is constant. The space $F_{(e, e')}$ is homotopy equivalent to the loop space $\Omega(\RP^d)$. Since $d\ge 2$, we have $\pi_1(\RP^d)=\Z_2$ and hence the space $F_{(e, e')}$ has two connected components. We shall denote by 
\begin{eqnarray}
\widehat {E^I_B}=E^I_B/\sim
\end{eqnarray}
the space of all equivalence classes with respect of the equivalence relation $\gamma\sim \gamma'$ iff $\gamma, \gamma'\in F_{(e, e')}$ and 
$\gamma$ is homotopic to $\gamma'$ relative $\partial I$. In other words, the space $\widehat {E^I_B}$ is obtained from $E^I_B$ by collapsing to a point every connected component of the fibre. The map $\Phi$ (which appears in the diagram below) is the quotient map. 
The map $\widehat\Pi: \widehat E^I_B\to E^2_B$ is well-defined and we have the commutative diagram
\begin{eqnarray}\label{diag:3}
\xymatrix{E^I_B\ar^{\Phi}[r] \ar[rd]_\Pi & \widehat E^I_B\ar[r]^{\Psi}\ar[d]^{\hat \Pi} & \widehat E^2_B\ar[dl]^{\hat q}\\
& E^2_B. &
}
\end{eqnarray}
 The space $\widehat{E^2_B}$ which appears in this diagram is the quotient space of $\tilde E^2_B$ with respect to the diagonal action of the group $G=\{1, -1\}$, see \S\ref{sec:43}. The map $\hat q$ is the quotient map $\widehat{ E^2_B}=\tilde E^2_B/G\to \tilde E^2_B/(G\times G)=E^2_B$. 
 The map $\hat q$ is a 2-fold covering. 
 
 The map $\Psi: \widehat{E^I_B}\to \widehat{E^2_B}$ (see the diagram above) is defined as follows. Given a point $[\gamma]\in F_{(e, e')}$, i.e. $\gamma\in E^I_B$ satisfies $\gamma(0)=e$ and $\gamma(1)=e'$, and $p\circ \gamma: I\to B$ is a constant map. There exists a lift $\tilde \gamma: I\to \tilde E$ into the unit sphere bundle and we set $\Psi[\gamma]$ to be the $G$-orbit of the pair $(\tilde \gamma(0), \tilde\gamma(1))\in \tilde E^2_B$ with respect to the diagonal action of 
 $G=\{1, -1\}$. The lift $\tilde \gamma$ is not unique, but the alternative lift is $-\tilde\gamma$, and the pair $(-\tilde \gamma(0), -\tilde\gamma(1))$ lies in the $G$-orbit of the pair $(\tilde \gamma(0), \tilde\gamma(1))$. 
 Hence, we see that the definition of $\Psi$ is correct.  
 \begin{corollary}\label{cor:29} Let $p: E\to B$ be the projectivisation of a vector bundle $\xi$ of rank $d+1$ where $d\ge 2$, over a finite CW-complex. 
 (1) The map $\Psi: \widehat{E^I_B}\to \widehat{E^2_B}$ is a homeomorphism, it establishes an isomorphism between the 2-fold coverings 
 $\hat \Pi: \widehat {E^I_B}\to E^2_B$ and $\hat q: \widehat{E^2_B}\to E^2_B$.
 (2) The local coefficient system $\tilde H_0(F_{(e, e')})$, formed by the reduced 0-dimensional homology of the fibres of the bundle $\Pi: E^I_B\to E^2_B$, is naturally isomorphic  (via the homomorphism induced by the map $\Psi\circ\Phi$)
to the local system formed by the reduced 0-dimensional homology of the fibres of  $\hat q: \widehat{E^2_B}\to E^2_B$. (3) The homological obstruction for a section of $\Pi: E^I_B\to E^2_B$ coincides with the primary obstruction for a section of the 2-fold covering 
 $\hat q: \widehat{E^2_B}\to E^2_B$. 
 \end{corollary}
 \begin{proof} Statement (1) and (2) are obvious: the map $\Psi$ is a bijection on each fibre and hence it is bijective. The inverse $\Psi^{-1}$ is continuous due to compactness of the spaces $\widehat{E^I_B}$ and $\widehat{E^2_B}$. Statement (3) follows directly from the definition of the homological obstruction, see \cite{S}, chapters 1 and 2. Recall that the homological obstruction is a class $\mathfrak o\in H^1(E^2_B; {\tilde H}_0(F))$ where $F$ is the fibre, in our case a two-point set, and ${\tilde H}_0(F))$ is viewed as a local coefficient system. The homological obstructions for the fibrations $\Pi$, $\hat \Pi$ 
  and $\hat q$ which appear in diagram \ref{diag:3} are equal. 
 \end{proof}
 
 \begin{corollary}
 The parametrized topological complexity $\TC[p:E\to B]$ attains its maximal value $\dim E^2_B=n+2d$ if and only if the sectional category of the 2-fold cover  $\hat q: \widehat{E^2_B}\to E^2_B$ equals its maximal value $\dim E^2_B=n+2d$, where $n=\dim B$. 
 \end{corollary}
 \begin{proof} 
By obstruction theory (see \cite{S} Theorem 1) the parametrised topological complexity $\TC[p:E\to B]$ and the sectional category $ \secat[\hat q: \widehat {E^2_B}\to E^2_B]$ are maximal and equal $n+2d$ if and only if the power of the homological obstruction 
$$\mathfrak o^{n+2d}\not=0\in H^{n+2d}(E^2_B;  (\tilde H_0(F))^{\otimes n+2d})$$
is nonzero. The result now follows from Corollary \ref{cor:29}.
 \end{proof}

Combining Lemma \ref{lm:21} and Theorem \ref{maxsecatmanifold}, we see that Theorem \ref{thm:sharp} follows once we show that the height of the class 
$\alpha=\vv_L+\vv_R$ is always less or equal than $n+2d-1$, i.e. 
\begin{eqnarray}\label{eq:35}
(\vv_L+\vv_R)^{n+2d}=0\in H^\ast(E^2_B;\Z_2). 
\end{eqnarray}
This is established in Theorem \ref{thm:33} below. 

For a rank $d+1$ vector bundle  $\xi$ over a CW-complex  $B$ we shall consider together with the Stiefel - Whitney classes $w_i=w_i(\xi)\in H^i(B;\Z_2)$ the dual Stiefel - Whitney classes $\bar w_i=\bar w_i(\xi)$, which are defined by the relation 
\begin{eqnarray}\label{eq:36}
w\cdot \bar w =1,
\end{eqnarray}
where $w=1+w_1+w_2+\dots$ and $\bar w = 1+\bar w_1+\bar w_2+\dots$ are the total Stiefel - Whitney and total dual Stiefel - Whitney classes. The equality (\ref{eq:36}) is understood in the ring $\prod_{i\ge 0} H^i(B;\Z_2)$ of formal power series.

For $i=0, 1, \dots, d+1$ let $Q_i(x)$ denote the following polynomial of degree $i$:
$$
Q_i(x)=x^i+w_1x^{i-1}+w_2x^{i-2}+\dots+ w_i, \quad Q_i \in H^\ast(B;\Z_2)[x].
$$
\begin{lemma}\label{lm:32}
For any $i>0$ in the ring $H^\ast(E^2_B;\Z_2)$ one has the following relations for the powers of the left and the right enhancements
\begin{eqnarray}\label{eq:37a}
\vv_L^{d+i}=\bar w_i Q_d(\vv_L) + \bar w_{i+1} Q_{d-1}(\vv_L) +\dots
\end{eqnarray}
and similarly 
\begin{eqnarray}\label{eq:38a}
\vv_R^{d+i}=\bar w_i Q_d(\vv_R) + \bar w_{i+1} Q_{d-1}(\vv_R) +\dots
\end{eqnarray}
\end{lemma}
\begin{proof}
Let ${\mathcal {V}}_L$ denote the class ${\mathcal {V}}_L=1+\vv_L +\vv_L^2+ \vv_L^3+\dots$ understood as an element of the ring of formal power series 
$$\mathcal R= \prod_{i\ge 0} H^\ast(E^2_B;\Z_2).$$ 
The ring $\prod_{i\ge 0} H^\ast(B;\Z_2)$ is naturally a subring of $\mathcal R$. Consider in $\mathcal R$ the product $w\cdot \VV_L$. For the $i$-th degree component of this product,
using the relation (\ref{eq:rel}),
we have
\begin{eqnarray}\label{eq:39a}
(w\cdot \VV_L)_i=\left\{
\begin{array}{lll}
Q_i(\vv_L)&\mbox{for}& i\le d,\\
0&\mbox{for}& i> d.
\end{array}
\right.
\end{eqnarray}
We see that the product $w\cdot \VV_L$ is a \lq\lq polynomial\rq\rq\,  of degree $d$.
Multiplying it by $\bar w$ in the ring 
$\mathcal R$ we obtain  
\begin{eqnarray}
\VV_L=\bar w\cdot (w\cdot \VV_L).
\end{eqnarray}
From this, taking into account (\ref{eq:39a}) we obtain (\ref{eq:37a}) and similarly (\ref{eq:38a}).
\end{proof}
%
%
%
%
%
%
Lemma \ref{lm:32} leads to the following Corollary:

\begin{corollary}\label{cor:33}
Let $m\ge 0$ be the maximal integer such that the dual Stiefel - Whitney class $\bar w_m=\bar w_m(\xi) \not=0\in H^m(B;\Z_2)$ is nonzero. Then the highest nonzero power of the class $\vv_L\in H^1(E^2_B;\Z_2)$ is $\vv_L^{m+d}$; moreover,
in the group $H^{m+d}(E^2_B;\Z_2)$ one has the equality
\begin{eqnarray}\label{eq:37}
\vv_L^{m+d}=\bar w_m\cdot \left[\vv_L^d+w_1\vv_L^{d-1}+w_2\vv_L^{d-2}+\dots+w_d\right].
\end{eqnarray}
Similarly, the highest nonzero power of the class $\vv_R$ is $\vv_R^{m+d}$, and
in the group $H^{m+d}(E^2_B;\Z_2)$ one has
\begin{eqnarray}\label{eq:38}
\vv_R^{m+d}=\bar w_m\cdot \left[\vv_R^d+w_1\vv_R^{d-1}+w_2\vv_R^{d-2}+\dots+w_d\right].
\end{eqnarray}
\end{corollary}
%

Now we are in position to proof the following result:

\begin{theorem}\label{thm:33}
Let $\xi$ be a vector bundle of rank $d+1\ge 3$ over a finite CW-complex $B$ of dimension $n=\dim B$. 
Let $p:E\to B$ be the bundle of real projective spaces $\RP^d$ obtained as projectivization of $\xi$. 
Let $\vv_L, \vv_R\in H^1(E^2_B;\Z_2)$ be the left and the right enhancements, see (\ref{eq:vlvr}). Then 
\begin{eqnarray}\label{eq:?}
(\vv_L+\vv_R)^{n+2d}=0.
\end{eqnarray}
In other words, $$\h(\vv_L+\vv_R)\le n+2d-1.$$
\end{theorem}
\begin{proof}
Let $m$ denote the highest dimension in which $\bar w_m(\xi)\not=0\in H^m(B;\Z_2)$. Clearly, $m\le n$. 

Assume first that $m=n$. Then 
\begin{eqnarray}\label{eq:41}
(\vv_L+\vv_R)^{m+2d} =\sum_{i=0}^m \binom {m+2d}{d+i} \vv_L^{d+i}\vv_R^{m+d-i}.
\end{eqnarray}
Using (\ref{eq:38}) we see that 
$$\vv_L^d\vv_R^{m+d}=\vv_L^d \bar w_m \left[\vv_R^{d} +w_1\vv_R^{d-1}+\dots\right]=\bar w_m\vv_L^d\vv_R^d$$
and similarly 
$\vv_L^{m+d}\vv_R^{d}=w_m\vv_L^d\vv_R^d.$
Hence the terms with $i=0$ and $i=m$ in the sum (\ref{eq:41}) are equal and they cancel each other. 
If $1\le i\le m-1$ then the product $\vv_L^{d+i}\vv_R^{m+d-i}$ equals $\bar w_i\bar w_{m-i} \vv_L^d\vv_R^d$ as follows from Lemma \ref{lm:32}. 
 Thus the terms with the indexes $i$ and $m-i$ cancel each other. If $m=2k$ is even then the central term appears with the coefficient 
 $\binom {m+2d}{d+k}$ which is always even. 

Consider now the case $m<n$.  
If $2m<n$ then in each product of the form $\vv_L^i\vv_R^{n+2d-j}$ one has either $j>m+d$ or $n+2d-j>m+d$ and hence the product vanishes due to Corollary \ref{cor:33}.
Finally we shall assume that $m<n\le 2m$. Then 
\begin{eqnarray}\label{eq:41b}
(\vv_L+\vv_R)^{n+2d} =\sum_{j=n-m}^{m} \binom {n+2d}{d+j} \vv_L^{d+j}\vv_R^{n+d-j}.
\end{eqnarray}
In this sum the lower limit $j=n-m$ does not exceed the upper limit $j=m$.  
Using Lemma \ref{lm:32} we have 
$
\vv_L^{d+j}\vv_R^{n+d-j} = \bar w_j \bar w_{n-j} \vv_L^d\vv_R^d
$
and therefore 
$$
(\vv_L+\vv_R)^{n+2d} = \left(
\sum_{j=n-m}^{m} \binom {n+2d}{d+j}\bar w_j \bar w_{n-j}
\right)\cdot \vv_L^d\vv_R^d.
$$
Since $\binom{n+2d}{d+j} = \binom{n+2d}{d+(n-j)}$, the terms of this sum with $j$ and $n-j$ cancel each other. If $n=2k$ is even and $j=k$ then $j=n-j$ and the binomial coefficient $\binom{n+2d}{d+k}$ is even, i.e. it vanishes modulo 2. 
This completes the proof.
\end{proof}

\section{Examples}\label{sec:6}

\subsection{} Let $B=\RP^2\times \RP^2\times \RP^2$ and let $a,b,c\in H^1(B;\Z_2)$ denote the generators of the corresponding factors.  Let $\xi$ be the vector bundle of rank $3$ obtained as the Whitney sum of the canonical lines bundles associated to $a$, $b$ and $c$. We calculate below the parametrized topological complexity 
of the bundle of real projective spaces $p: E\to B$ with fibre $\RP^2$ over $B$ using the results of the previous sections. 

We note that the total Stiefel - Whitney class of $\xi$ is given by
$$
w= (1+a)(1+b)(1+c)=1+w_1+w_2+w_3\, \in H^\ast(B;\Z_2)
$$
and the dual Stiefel - Whitney class is 
$$
\bar w= (1+a+a^2)(1+b+b^2)(1+c+c^2) \, \in H^\ast(B;\Z_2).
$$
The highest nontrivial class $\bar w_m$ is $\bar w_6=(abc)^2\in H^6(B;\Z_2)$. 

Let $E^2_B\subset E^2$ be the space of pairs of lines lying in the same fibre, as described in Lemma \ref{lm:coh1}. 
Let $\vv_L, \vv_R\in H^1(E^2_B;\Z_2)$ be the left and the right enhancements. Applying Corollary \ref{cor:33} we find
\begin{eqnarray}\label{eq:41d}
\vv_L^8=(abc)^2\vv_L^2, \quad\mbox{and}\quad \vv_R^8=(abc)^2\vv_R^2
\end{eqnarray}
and $\vv_L^9=0=\vv_R^9$, i.e. $\h(\vv_L)=\h(\vv_R)=8.$ Thus, we obtain
$$
(\vv_L+\vv_R)^8= \vv_L^8+\vv_R^8= (abc)^2 (\vv_L^2+\vv_R^2)\not=0
$$
and 
$$
(\vv_L+\vv_R)^9=(abc)^2 (\vv_L^2+\vv_R^2)(\vv_L+\vv_R)= (abc)^2 (\vv_L^2\vv_R+\vv_R^2 \vv_L)\not=0
$$
is nonzero, as one sees by applying Lemma \ref{lm:coh1}.
On the other hand, using (\ref{eq:41d}) we find
$$
(\vv_L+\vv_R)^{10} = (\vv_L^8+\vv_R^8)\cdot (\vv_L^2+\vv_R^2)=(abc)^2\cdot (\vv_L^4+\vv_R^4)=0.
$$
The vanishing of the class $(abc)^2\cdot (\vv_L^4+\vv_R^4)\in H^{10}(E^2_B;\Z_2)$ can be seen using Lemma \ref{lm:32}; indeed, the class $(abc)^2\in H^6(B;\Z_2)$ is a top degree class of $B$ and the expansions of $\vv_L^4$ and $\vv_R^4$ involve classes of the base of positive degree. 

As $\dim B=6$ and $d=2$, we may apply Corollary \ref{cor:13} combined with Theorem \ref{thm:sharp} to conclude that $\TC[p:E\to B]=9$. 

\subsection{} Consider the circle $B=S^1$ and the vector bundle $\xi$ over $B$ of rank $d+1$ where 
$\xi=\eta\oplus \epsilon \oplus \epsilon \dots\oplus \epsilon$ where $\eta$ is a non-orientable line bundle and $\epsilon$ is the trivial line bundle. 
We shall assume that $d\ge 2$. 
We are interested in the projectivization $p: E\to B$ of $\xi$ which is a locally trivial bundle with fibre $\RP^d$ over the circle. 
The total Stiefel - Whitney class $w=w(\xi)=1+\beta$, where $\beta\in H^1(B;\Z_2)$ is the generator. The relations (\ref{eq:rel}) for the left and right enhancements in this case read
\begin{eqnarray}
\vv_L^{d+1}=\beta \vv_L^d, \quad \vv_R^{d+1}=\beta \vv_R^d
\end{eqnarray}
and $\vv_L^{d+2}=0=\vv_R^{d+2}.$
The upper bound of Theorem \ref{thm:sharp} gives 
\begin{eqnarray}
\TC[p: E\to B] \le 2d,
\end{eqnarray}
and to match it with the lower bound of Corollary \ref{cor:13} we need to have $(\vv_L+\vv_R)^{2d}\not=0$. 
The binomial expansion of $(\vv_L+\vv_R)^{2d}$ contains the terms 
$$
\binom{2d}{d-1}\vv_L^{d-1}\vv_R^{d+1} +\binom{2d}{d+1}\vv_L^{d+1}\vv_R^{d-1}= \binom{2d}{d-1}\beta\left[\vv_L^{d-1}\vv_R^d+\vv_L^d\vv_R^{d-1}\right].
$$
We see that the parametrized topological complexity $\TC[p: E\to B]$ is maximal and equals $2d$ if the binomial coefficient $\binom{2d}{d-1}$ is odd. 

In general, if $a$ and $b$ are two integers with their binary expansions  $a=a_1a_2\dots a_k$ and $b=b_1b_2\dots b_{k}$, where $a_i, b_i\in\{0,1\}$,
then $$\binom a b=\prod_{i=1}^k \binom {a_i}{b_i}\mod 2.$$ Here one understands that $\binom 0 0=\binom 1 1=\binom 1 0=1$ and $\binom 0 1=0$.
Thus, the binomial coefficient $\binom a b$ is odd iff $a_i=0$ implies $b_i=0$, i.e. if the binary expansion of $b$ has zeros on all positions where the expansion of $a$ has zeros. 

Using the above general criterion, it is easy to see that the binomial coefficient $\binom{2d}{d-1}$ is odd when $d+1$ is a power of 2. 
This leads to the following Corollary: 
\begin{corollary}
If $d+1$ is a power of 2 then one has $\TC[p:E\to B]=2d$  in the above example. 
\end{corollary}
In particular, one may take $d=15$ and get $\TC[p:E\to B]=30$ for the associated bundle of real projective spaces $\RP^{15}$ over the circle. 
It is interesting to compare this result with fact that $\TC(\RP^{15})= 22$, see (\ref{eq:rp15}). 
This example illustrates a surprising phenomenon that, even in the case of a one-dimensional base, the parametrized topological complexity may significantly exceed the usual topological complexity.

\begin{appendix}\label{app:1}

\section{Circle bundles and vector bundles of rank 2}\label{sec:app1}

The proof of the following statement uses arguments of \cite{Ghys}, \S 4. 

\begin{lemma}\label{lm:13}
If $B$ is a finite CW complex, then every locally trivial bundle $p:E\to B$ with fibre the circle is isomorphic to the unit sphere bundle of a rank 2 vector bundle  
$\xi$ over $B$. Moreover, such bundle $\xi$ is unique up to isomorphism. 
\end{lemma}
\begin{proof}
Let $H(S^1)=H_+(S^1)\sqcup H_-(S^1)$ be the group of homeomorphisms of the circle $S^1$ with $H_\pm(S^1)\subset H(S^1)$ denoting the subgroups of homeomorphisms preserving (reversing) the orientation. Clearly, one has $O(2)\subset H(S^1)$ and 
$SO(2)\subset H_+(S^1)$.  

\begin{lemma} \label{lm:14} 
The induced map of classifying spaces
\begin{eqnarray}
BO(2) \to BH(S^1) 
\end{eqnarray}
is a fibration with contractible fibre.  Therefore, for any finite CW complex $B$ any continuous map $B\to BH(S^1)$ admits a unique up to homotopy lift to a map 
$B\to BO(2)$. 
\end{lemma}

Clearly Lemma \ref{lm:14} implies Lemma \ref{lm:13}. 

Proposition 4.2 from \cite{Ghys} states that the inclusion $SO(2)\to H_+(S^1)$ is a homotopy equivalence, but this fact is not enough for our purposes. 

We show below (using the method of \cite{Ghys})
that 
the group $H(S^1)$, viewed as a left $O(2)$-space, is the total space of a principal $O(2)$-bundle with contractible orbit space $H(S^1)/O(2)$ and 
therefore the canonical map $BO(2) \to BH(S^1)$ is a fibration with contractible fibre $H(S^1)/O(2)$. 

Let ${\mathbf e}: \R\to S^1$ denote the universal covering of the circle, ${\mathbf e}(t)= \exp(2\pi i t)$, where $t\in \R$. 
Any homeomorphism $f: S^1\to S^1$ admits a lift $\tilde f$ into the universal covers
$$
\xymatrix{
\R\ar[r]^{\tilde f}\ar[d]_\e&\R\ar[d]^\e\\
S^1\ar[r]_f& S^1
}
$$
and $\tilde f$ is unique, up to addition of an integer. Explicitly, $f(\exp(2\pi i t))=\exp(2\pi i \tilde f(t))$.
The lift $\tilde f:\R\to \R$ satisfies either $\tilde f(t+1)=\tilde f(t)+ 1$, or $\tilde f(t+1)=\tilde f(t)-1$, for all $t\in \R$, 
depending on whether $f$ preserves or reverses the orientation of the circle; it follows that the lift $\tilde f$ of every homeomorphism $f:S^1\to S^1$ is a homeomorphism $\tilde f: \R\to \R$, it is monotone increasing iff $f$ preserves the orientation and it is monotone decreasing otherwise. 

The lift $\tilde R_c$ of a rotation $R_c\in SO(2)$, where $R_c(e^{2\pi i t})=e^{2\pi i (t+c)}$, has the form $\tilde R(t)=t+c$, where $c\in \R$ is a constant. 
The lift of a reflection $Q_c\in O(2)$, where $Q_c(e^{2\pi i t})=e^{2\pi i (-t+c)}$, has the form $\tilde Q_c(t) = -t+c$. Note the formulae
\begin{eqnarray}\label{eq:formulae}
Q_{c'}\circ R_c= Q_{c'-c}= R_{-c}\circ Q_{c'} \quad \mbox{and}\quad Q_{c'}\circ Q_c=R_{c'-c}.
\end{eqnarray}

Let $F\subset C(\R)$ denote the space of all continuous function $g:\R\to \R$ having the following properties:
\begin{enumerate}
\item[{(a)}] $g(t+1)=g(t)$;
\item[{(b)}] $\int_0^1g(t)dt=0; $
\item[{(c)}] for any $t<t'$ one has $\frac{g(t')-g(t)}{t'-t}>-1$. 
\end{enumerate}
The space $F$ is convex and hence is contractible. 

We claim that {\it there exists an $O(2)$-equivariant homeomorphism} 
\begin{eqnarray}\label{eq:eq}
\psi: H(S^1) \to O(2)\times F.
\end{eqnarray}\label{eq:split}
For $f\in H_+(S^1)$ the lift $\tilde f$ can be written uniquely in the form 
\begin{eqnarray}
\tilde f(t)=t+c+g(t),\quad t\in \R,
\end{eqnarray} 
where $c\in \R$ is a constant (defined up to adding an integer)
and $g\in F$. The condition (c) is equivalent for $\tilde f$ to be monotone. The map (\ref{eq:eq}) sends $f\in H_+(S^1)$ to the pair 
$$\psi(f) =(R_c, g)\in O(2)\times F,$$ 
where 
$R_c\in SO(2)$ is the rotation by the angle $2\pi  c$. If $f\in H_-(S^1)$ is a homeomorphism reversing the orientation then the lift $\tilde f$ can be uniquely written in the form 
\begin{eqnarray}
\tilde f(t)=-t-c-g(t),\quad t\in \R,
\end{eqnarray} 
where $g\in F$, and we define 
$$
\psi(f) =(Q_c, g)\in O(2)\times F,
$$ 
where $Q_c$ is the reflection described above. 

It is obvious that $\psi$ is a homeomorphism and using formulae (\ref{eq:formulae}) one checks that $\psi$ is $O(2)$-equivariant, i.e. if $\psi(f) =(T, g)$, where $T\in O(2)$ and $g\in F$, then 
$$\psi(S\circ f) = (S\circ T, g)\quad\mbox{ for every}\quad  S\in O(2).$$ 
Thus we see that the subgroup $O(2)\subset H(S^1)$ is admissible, i.e. the quotient map $H(S^1)\to H(S^1)/O(2)$ is a principal bundle (see p. 5 of \cite{Mitchell}) and fibre of the canonical fibration $BO(2)\to BH(S^1)$ is homeomorphic to $H(S^1)/O(2)= F$
hence is contractible. This completes the proof of Lemma \ref{lm:13}.
\end{proof}

%
%
%
%
%
%
%
%

\section{Sectional category of 2-fold covering}\label{app:2}

Let $\phi:\tilde Y\to Y$ be a 2-fold cover of a finite CW-complex $Y$, where $\tilde Y$ and $Y$ are connected. Let 
$\alpha \in H^1(Y;\Z_2)$, \, $\alpha\not=0$, be the characteristic class of $\phi$.  This means that 
the kernel of the homomorphism $\pi_1(Y)\to \Z_2$ determined by $\alpha$ coincides with the subgroup $\pi_1(\tilde Y)\subset \pi_1(Y)$. 
For the sectional category of the covering $\phi: \tilde Y\to Y$ one has
\begin{eqnarray}\label{standardbounds}
\h(\alpha)\le \secat[\phi:\tilde Y\to Y]={\mathfrak {genus}}(\alpha) \le \dim Y.
\end{eqnarray}	
The left inequality is the standard cohomological lower bound using the fact that $\phi^\ast(\alpha)=0$. 	

Lemma \ref{doublecover} and Lemma \ref{doublecovera} stated below give stronger upper bounds for the sectional category and for the genus of $\alpha$, compared to (\ref{standardbounds}). Roughly, these statements claim that the sectional category 
$\secat[\phi:\tilde Y\to Y]$ attains the upper bound $n=\dim Y$ only when $\h(\alpha)=n$, i.e. $\alpha^n\not=0$. 

\begin{lemma}\label{doublecover} Let $Y$ be a finite cell complex of even dimension $n\ge 2$, 
and let $\alpha\in H^1(Y;\Z_2)$ be a nonzero cohomology class. Assume that the top integral cohomology group $H^n(Y;\Z)$ 
has no direct summands of type $\Z_{2^k}$ with $k>1$. If $\h(\alpha)\le n-1$, i.e. $\alpha^n=0$, then 
\begin{eqnarray}
{\mathfrak {genus}}(\alpha) = \secat[\phi:\tilde Y\to Y] \, \le n-1.
\end{eqnarray}
Here $\phi: \tilde Y\to Y$ denotes the 2-fold covering corresponding to the class $\alpha$. 
\end{lemma}

To state Lemma \ref{doublecovera} which is the odd-dimensional version of Lemma \ref{doublecover}, we need to introduce the following notation. 
Given a cohomology class $\alpha\in H^1(Y;\Z_2)$, let $I_\alpha$ denote the local coefficient system of groups $\Z$ over $Y$ such that its monodromy along any loop $\gamma$ in $Y$ equals 
$$(-1)^{\alpha(\gamma)}\in \{1, -1\}.$$ 
Here $\alpha(\gamma)\in \Z_2=\{0,1\}$ and $(-1)^0=1$ and $(-1)^1=-1$. In other words, the monodromy of $I_\alpha$ along a loop $\gamma$ is nontrivial iff the value of the class $\alpha$ on $\gamma$ is nonzero. 

\begin{lemma} \label{doublecovera} 
Let $Y$ be a finite cell complex of odd dimension $n\ge 3$, 
and let $\alpha\in H^1(Y;\Z_2)$ be a nonzero cohomology class. 
Assume that the cohomology group $H^n(Y;I_\alpha)$ 
has no direct summands of type $\Z_{2^k}$ with $k>1$. 
If $\h(\alpha)\le n-1$ then
\begin{eqnarray}
{\mathfrak {genus}}(\alpha) = \secat[\phi:\tilde Y\to Y] \, \le n-1.
\end{eqnarray}
\end{lemma}


\begin{proof}[Proof of Lemma \ref{doublecover} and Lemma \ref{doublecovera}] According to a theorem of A. Schwarz \cite{S}, one has $\secat[\phi:\tilde Y\to Y]\le  n-1$ if and only if the $n$-fold fibrewise join 
$$\phi_n: \ast^n_Y\tilde Y\to Y$$
of $n$ copies of the covering $\phi: \tilde Y\to Y$ admits a continuous section. 
The fibre of $\phi_n$ is the join $S^0\ast S^0\ast\dots \ast S^0= (S^0)^{\ast n}$ of $n$ copies of 
$S^0$, it is homeomorphic to $S^{n-1}$ and is $(n-2)$-connected. 
The primary obstruction $\mathfrak o_n$ for a section of $\phi_n$ lies in the cohomology group 
\begin{eqnarray}\label{eq:obstr}
\mathfrak o_n\in H^n(Y; \pi_{n-1}((S^0)^{\ast n}))\simeq H^n(Y; H_{n-1}((S^0)^{\ast n};\Z)),
\end{eqnarray} 
where the coefficients $H_{n-1}((S^0)^{\ast n};\Z)\simeq \Z$ are understood as a local coefficient system. In fact, for dimensional reasons, ${\mathfrak o}_n$ is the only obstruction, i.e. the section of $\phi_n$  exists iff ${\mathfrak o}_n=0$. 

According to Schwarz \cite{S}, the primary obstruction $\mathfrak o_n$ can be computed as follows. 

Let $G=\{1, t\}$ denote the multiplicative cyclic group of order 2, i.e. $t^2=1$. 
The class $\alpha\in H^1(Y;\Z_2)$ determines a group homomorphism $\pi_1(Y)\to G$, where $g\mapsto t^{\alpha(g)}$, which allows to view the group ring $\Z[G]$ as a local coefficient system over $Y$. If $\epsilon : \Z[G]\to \Z$ is the augmentation and $I=\ker \epsilon$ denotes the augmentation ideal, then 
\begin{eqnarray}\label{eq:exsec}
0\to I\to \Z[G]\stackrel{\epsilon}\to \Z\to 0
\end{eqnarray}
is an exact sequence of local systems over $Y$. 

As an abelian group, the ideal $I$ is isomorphic to $\Z$, it has generator $t-1$ and the action of $\pi_1(Y)$ on the generator 
$(t-1)\in I$ is as follows: the action of a group element
 $g\in \pi_1(Y)$ on $(t-1)$ equals $$g\cdot (t-1) = t^{\alpha(g)}\cdot (t-1).$$ Explicitly, this means that 
if $g\in \pi_1(\tilde Y)\subset \pi_1(Y)$, i.e. $\alpha(g) =0$, then $g\cdot (t-1)=(t-1)$ and if $g\notin \pi_1(\tilde Y)$, i.e. 
$\alpha(g) =1$, then $g\cdot(t-1)=t(t-1)=-(t-1)$. Thus, we see that {\it the local system $I$ is isomorphic to $I_\alpha$}. 

The short exact sequence (\ref{eq:exsec}) gives a long exact sequence
\begin{eqnarray}\label{eq:58}
H^0(Y;I)\to H^0(Y;Z[G])\stackrel{\epsilon}\to H^0(Y;\Z)\stackrel{\delta}\to H^1(Y;I).
\end{eqnarray}
Here $\delta$ is the Bockstein homomorphism. 

The class ${\sf w}=\delta(1)\in H^1(Y;I)$ is {\it the homological obstruction}, in terminology of \cite{S}. 
According to the general theory, see \cite{W}, chapter 6, the cohomology class $\sf w$ is represented by a cross-homomorphism
$\psi: \pi_1(Y)\to I$, which is given by $\psi(g)=t^{\alpha(g)}-1\in I$. 
Hence 
\begin{eqnarray}\label{eq:59}
\psi(g) = \left\{
\begin{array}{lll}
0, &\mbox{if}& \alpha(g)=0,\\
t-1, &\mbox{if}& \alpha(g)=1.
\end{array}
\right.
\end{eqnarray}

Applying Theorem 3.2 from \cite{W}, chapter 6, we obtain that in the exact sequence (\ref{eq:58}) one has $H^0(Y;I)=0$ and $H^0(Y;\Z[G])\simeq \Z$. 
The latter group is generated by the class of the invariant element $1+t\in G$. Since $\epsilon(1+t)=2$ we conclude that the homological obstruction $\sf w$ has order 2, i.e. 
\begin{eqnarray}\label{eq:60}
2\cdot {\sf w} =0. 
\end{eqnarray}

Theorem 2 of Schwarz \cite{S} states that the local coefficient system $H_{n-1}((S^0)^{\ast n};\Z)$ is isomorphic to the tensor power 
$I^{\otimes n}=I\otimes I\otimes I\otimes\cdots\otimes I$ ($n$ factors)
 and the primary obstruction $\mathfrak o_n$ equals the $n$-th power of the homological obstruction $\sf w$, i.e.
\begin{eqnarray}
\mathfrak o_n= { \sf w}^n\in H^n(Y; I^{\otimes n}). 
\end{eqnarray}

We show below that ${\sf w}^n=0\in H^n(Y; I^{\otimes n})$ if and only if $\alpha^n=0\in H^n(Y;\Z_2)$, and therefore $\secat[\phi:\tilde Y\to Y]\le n-1$ if and only if 
$\alpha^n=0$.

Note that $I\otimes I\simeq \Z$, isomorphism of local coefficient systems. Hence the local system $I^{\otimes n}$ is the trivial local system $\Z$ for $n$ even, and $I^{\otimes n}=I=I_\alpha$ for $n$ odd. 

Next we consider the homomorphism $r: I\to \Z_2$ of the local systems (where $\Z_2$ is a constant local system) given by $r(t-1)=1$. Every element of $I$
has the form $x=k(t-1)$ with $k\in \Z$ and $r(x)\equiv k \mod 2$. The induced homomorphism 
$$
r_\ast: H^1(Y;I)\to H^1(Y;\Z_2)
$$
satisfies $r_\ast({\sf w})=\alpha$ as follows from (\ref{eq:59}). Thus, we see that ${\sf w}^n=0$ implies $\alpha^n=0$. 

To show the inverse implication, we first assume that $n\ge 2$ is even. Then $I^{\otimes n}=\Z$ and ${\sf w}^n\in H^n(Y;\Z)$. 
We have the exact sequence 
$$H^n(Y;\Z)\, \stackrel{\times 2}\to \, H^n(Y;\Z)\, \stackrel{r_\ast}\to \, H^n(Y;\Z_2)$$
with $r_\ast({\sf w}^n)=\alpha^n$. We see that if $\alpha^n=0$ then the class ${\sf w}^n$ must be divisible by $2$. On the other hand, we have $2 \cdot {\sf w}^n=0$, because of (\ref{eq:60}). Our assumption on the absence of summands of type $\Z_{2^k}$ with $k>1$ in the group $H^n(Y;\Z)$ implies that 
${\sf w}^n=0$. 

Consider now the case when $n$ is odd. Then $I^{\otimes n}\simeq I$
and the exact sequence of local coefficient systems over $Y$
$$
0\to I\stackrel{\times 2}\to I\stackrel{r} \to \Z_2\to 0.
$$
gives the cohomological sequence 
$$
H^n(Y;I)\, \stackrel{\times 2}\to \, H^n(Y;I)\, \stackrel{r_\ast}\to \, H^n(Y;\Z_2).
$$
We see  that ${\sf w}^n \in H^n(Y;I)$ and 
$r_\ast({\sf w}^n)=\alpha^n$. If $\alpha^n=0$ then the class ${\sf w}^n$ must be divisible by $2$. 
On the other hand this class has order $2$ and we obtain ${\sf w}^n=0$, using our assumption about the absence of summand of type $\Z_{2^k}$ with $k>1$ in the cohomology group $H^n(Y;I)$. This completes the proof. 
\end{proof}

The following result is a generalisation of Theorem 3.5 of \cite{Ber}. 

\begin{theorem} \label{maxsecatmanifold} Let $Y$ be a closed manifold of dimension $n\ge 2$. Then, for a nonzero cohomology class $\alpha\in H^1(Y;\Z_2)$, one has ${\mathfrak {genus}}(\alpha)=n$ if and only if $\alpha^n\neq 0$. 
\end{theorem}

\begin{proof} If the dimension $n$ is even we apply Lemma \ref{doublecover}. The top cohomology group $H^n(Y;\Z)$ is either $\Z$ or $\Z_2$ and hence the assumption of Lemma \ref{doublecover} is fulfilled. 

In the case when the dimension $n$ is odd we apply Lemma \ref{doublecovera}. We need to verify that the cohomology group $H^n(Y; I_\alpha)$ 
has no direct summands of type $\Z_{2^k}$ with $k>1$. Applying the Poincar\'e duality, we get 
$H^n(Y;I_\alpha)\simeq H_0(Y; I_\alpha\otimes \tilde \Z)$ where 
where $\tilde \Z$ is the orientation local system of the manifold $Y$. We observe that $\tilde \Z$ coincides with the local coefficient system $I_{\alpha'}$ where 
$\alpha'=w_1(Y) \in H^1(Y;\Z_2)$, the first Stiefel - Whitney class of $Y$. 
We find that $H^n(Y; I_\alpha)=H_0(Y;\Z)=  \Z$ iff $\alpha=w_1(Y)$ and $H^n(Y;I_\alpha)=\Z_2$ if $\alpha\not=w_1(Y)$. Hence, we see that the assumption of Lemma 
\ref{doublecovera} is satisfied. 

This completes the proof. 
\end{proof}

\end{appendix}

 \bibliographystyle{amsplain}

\begin{thebibliography}{99}
 


\bibitem{A} M. F. Atiyah, \textit{K-theory.} Notes by D. W. Anderson. Second edition. Advanced Book Classics. Addison-Wesley Publishing Company, Advanced Book Program, Redwood City, CA, 1989. xx+216 pp.

\bibitem{Ber} I. Berstein, \textit{On the Lusternik-Schnirelmann category of Grassmannians}, 
Math. Proc. Cambridge Philos. Soc. 79 (1976), no. 1, 129–134.

\bibitem{CVGCV} J. G. Carrasquel-Vera, J.M. García-Calcines, L. Vandembroucq, 
\textit{Relative category and monoidal topological complexity.} 
Topology Appl. {\bf{171}} (2014), 41–53.
 
 \bibitem{CFW} D.C. Cohen, M. Farber, S. Weinberger, \textit{Topology of Parametrized Motion Planning Algorithms}, SIAM J. of Applied Algebra and Geometry, {\bf 5}(2021), pp. 229--249. 
 
 \bibitem{CFW2} D.C. Cohen, M. Farber, S. Weinberger, \textit{Parametrized topological complexity of collision-free motion planning in the plane}, 
 Annals of Mathematics and Artificial Intelligence, {\bf 90} (2022), no. 10, pp. 999–1015. 
 
 \bibitem{CF} A. Costa, M. Farber, \textit{Motion planning in spaces with small fundamental groups}. Commun. Contemp. Math. 12 (2010), no. 1, 107–119.
 



 
 
 \bibitem{F} M. Farber, \textit{Topological complexity of motion planning}, Discrete Comput. Geom. 29 (2003), 211–221. 
 
\bibitem{FTY} {M. Farber, S. Tabachnikov, and S. Yuzvinsky,} Topological robotics: motion planning in projective spaces, \textit{Int. Math. Res. Not.} \textbf{34}, 1853--1870 (2003).
 
 \bibitem{F2} M. Farber, \textit{Invitation to topological robotics}, Zurich Lectures in Advanced Mathematics, EMS, 2008. 
 
 
 \bibitem{FP} M. Farber and A.K. Paul, \textit{Sequential parametrized topological complexity of sphere bundles}. Eur. J. Math. 10 (2024), no. 4, Paper No. 53, 29 pp.
 
 

\bibitem{FW} M. Farber and S. Weinberger, \textit{Parametrized motion planning and topological complexity}. In: Algorithmic Foundations of Robotics XV
Steven M LaValle et al editors,
Springer 2023, pp 1- 17.

\bibitem{FW2} M. Farber, S. Weinberger, \textit{Parametrized topological complexity of sphere bundles}. Topol. Methods Nonlinear Anal. 61 (2023), no. 1, 161–177. 


 \bibitem{JGC} J. M. Garc\'{\i}a-Calcines, \textit{A note on covers defining relative and sectional categories,} Topology Appl., 265 (2019), 106810.
 
 \bibitem{Ghys} E. Ghys, \textit{Groups acting on the circle}. Enseign. Math. (2) 47 (2001), no. 3-4, 329–407.


 \bibitem{Grant} M. Grant, \textit{Topological complexity, fibrations and symmetry}. Topology Appl. 159 (2012), no. 1, 88–97.
 
   \bibitem{Hu} D. Husemoller, \textit{Fibre bundles}. McGraw-Hill Book Co., New York-London-Sydney, 1966.
 
 
%
\bibitem{LV} S. M. LaValle, \textit{Planning algorithms}, Cambridge University Press, 2006. 
 
%
\bibitem{Mitchell} S. A. Mitchell, \textit{Notes on principal bundles and classifying spaces}, \newline
https://sites.math.washington.edu/~mitchell/Notes/prin.pdf
 
 \bibitem{S} A.S. Schwarz, \textit{The genus of a fibre space.}
Trudy Moscow Math Society {\bf{11}}(1962),  99 -- 126. 


\bibitem{W} G. W. Whitehead, \textit{Elements of homotopy theory}. Graduate Texts in Mathematics, 61. Springer-Verlag, New York-Berlin, 1978
 \end{thebibliography}

\end{document}